\documentclass{amsart}
\usepackage{amsfonts,amssymb}
\theoremstyle{plain}
\newtheorem{theorem}{Theorem}[section]
\newtheorem{proposition}[theorem]{Proposition}
\newtheorem{lemma}[theorem]{Lemma}
\newtheorem{corollary}[theorem]{Corollary}

\theoremstyle{definition}
\newtheorem{definition}[theorem]{Definition}

\newtheorem{remark}[theorem]{Remark}

\let\geq\geqslant
\let\leq\leqslant

\newcommand{\N}{\mathbb{N}}                   % natural numbers
\newcommand{\R}{\mathbb{R}}                   % real numbers
\newcommand{\C}{\mathcal{C}}
\newcommand{\sX}{\mathfrak{X}}                % stratification1
\newcommand{\W}{\mathcal{W}}                  % Whiteny condition
\newcommand{\Q}{\mathcal{Q}}
\newcommand{\WL}{\mathcal{WL}}
\newcommand{\T}{\mathcal{T}}

\begin{document}
\title{Definable triangulations with regularity conditions}
\author{Ma\l gorzata Czapla}
\address{Uniwersytet Jagiello\'nski, Instytut Matematyki, ul. \L ojasiewicza 6, 30-348 Krak\'ow, Poland}
\email{Malgorzata.Czapla@im.uj.edu.pl}
 %\thanks{Research was partially supported by the Natural Sciences and Engineering Research Council of Canada.}

\begin{abstract}
We prove that every definable set has a
definable triangulation which is locally Lipschitz and weakly bi-Lipschitz
on the natural simplicial stratification of a simplicial
complex. We also distinguish a class $\T$\ of regularity conditions and give a universal construction of a definable triangulation with a $\T$\ condition
of a definable set. This class includes the Whitney (B) and the Verdier conditions.
\end{abstract}
\maketitle

\section*{Introduction}

It has been known for more than 40 years - since papers of Whitney \cite{Wh} and \L ojasiewicz \cite{L1} - that analytic and semianalytic subsets of euclidean
spaces admit stratifications with Whitney regularity conditions, a result later generalized to subanalytic (\cite{Hi1}, \cite{LSW}) and finally to subsets 
definable in any o-minimal structures on $\R$\ (\cite{TL1}, \cite{TL2}). Since \L ojasiewicz's paper \cite{L2}, it has also been known that semianalytic 
and subsequently, subanalytic (\cite{Ha}, \cite{Hi2}, \cite{L3})
 and definable in o-minimal structures (\cite{vdD}) sets are triangulable.

A challenging problem stated by \L ojasiewicz and Thom was to combine the both results, i.e. to construct a triangulation of semi(sub)analytic sets
which is a stratification with Whitney conditions. A main difficulty was that the construction of Whitney stratification was by downward induction on dimension
in contrast to the triangulation which goes by upward induction on dimension. It was not clear how to overcome this divergence.

A first positive solution to the problem was given by Masahiro Shiota \cite{Sh1}. In his eight-page article concerning semialgebraic case, he proposed
a solution based on a technique of controlled tube systems developped in his book \cite{Sh2}. However, his proof is difficult to understand.

In the present article we give a direct constructive solution to the problem based on the theory of weakly Lipschitz mappings \cite{Cz} and on Guillaume
Valette's description of Lipschitz structure of definable sets \cite{Val}. Our solution is general in the sense that it concerns an arbitrary o-minimal
structure on the ordered field of real numbers $\R$\ (or even on any real closed field) and, moreover, in the sense that we describe a class $\T$\ of
regularity conditions including the Whitney and the Verdier conditions, such that for any condition $\Q$\ from $\T$\ a definable triangulation
with $\Q$\ condition is possible.

Roughly speaking, our final result (Theorem \ref{trm: triangulation with Q}) is derived from existence of a definable, locally Lipschitz, weakly bi-Lipschitz triangulation
(Theorem \ref{trm: Weakly_lipschitz_triang}), which in some sense (see Theorem \ref{trm: Q invariance}) preserves regularity conditions. To construct such a triangulation,
we first use Guillaume Valette's theorem (Theorem \ref{trm: GV2 bilipsch homeom}), which reduces the general case to the one, to which the classical construction of
triangulation can be applied.

A natural setting for our results is the theory of o-minimal structures (or more generally geometric categories), as presented in \cite{vdD}, \cite{DM}.
In the whole paper the adjective \textit{definable} (i.e. definable subset, definable mapping) will refer to any fixed o-minimal structure on
the ordered field of real numbers $\mathbb{R}$\ (or, more generally, on a real closed field) .

\section{Weakly Lipschitz mappings}
\label{sec: Weakly -lipschitz }

In this section we recall a notion of a weakly Lipschitz mapping and list its important properties.

We denote by $|\cdot|$\ the euclidean norm of $\R^{n}$. In the whole paper $C^{q}$\ denotes the class of smoothness of a mapping, so
$q\in\N$\ or $q\in\{\infty, \omega\}$, $q\geq 1$, unless it is said differently.
Now we remind briefly a notion of a $C^{q}$\ stratification.

\begin{definition} Let $A$\ be a subset of $\R^{n}$.
\textit{A} $C^{q}$\ \textit{stratification} of the set $A$\ is a (locally) finite family
$\sX_{A}$\ of connected $C^{q}$\ submanifolds of $\R^{n}$\ (called \textit{strata}),
such that \vskip 1 mm
\begin{tabular}{l}
$1)$\ \ $A=\bigcup\sX_{A}$\ ;\\
$2)$\ \ if\ $\Gamma_{1}, \Gamma_{2}\in\sX_{A}$, $\Gamma_{1}\neq\Gamma_{2}$\ then $\Gamma_{1}\cap\Gamma_{2}=\emptyset$; \\
$3)$\ \ for each $\Gamma\in\sX_{A}$\ the set
$(\overline{\Gamma}\setminus\Gamma)\cap A$\ is a union of
 some strata from $\sX_{A}$\ of dimension $<\dim\Gamma$.\\
\end{tabular}
\vskip 1 mm We say that the stratification $\sX_{A}$\ is
\textit{compatible with a family of sets} $B_{i}\subset A$, $i\in I$, if every set $B_{i}$\ is
a union of some strata of $\sX_{A}$.
\end{definition}

Actually, we will be interested only in finite stratifications.

\begin{definition}
\label{def:weakly_lipschitz-map} Let $A$\ be a subset of $\R^{n}$\ and let $\sX_{A}$\ be a finite $C^{q}$\
stratification of the set $A$. Consider a mapping $f: A\longrightarrow \R^{m}$.
We say that $f$\ is \textit{weakly Lipschitz of class }$C^{q}$\ \textit{on the stratification}\ $\sX_{A}$,
\ if for each stratum $\Gamma\in\sX_{A}$\ the restriction $f|_{\Gamma}$\ is of class $C^{q}$\ and
for any point $a\in\Gamma$\ there exists a neighbourhood $U_{a}$\ of $a$\ such that the mapping
$$\psi: (\Gamma\cap U_{a})\times ((A\setminus \Gamma)\cap U_{a})\ni(x,y)\longmapsto \frac{|f(x)-f(y)|}{|x-y|}\in\R$$
is bounded.
\vskip 1 mm
\end{definition}

\begin{remark} For another equivalent description see \cite{Cz} Definition 2.1.
\end{remark}

\begin{remark} If $f: A\longrightarrow \R^{m}$\ is weakly Lipschitz on a stratification $\sX_{A}$\ of the set $A$, then
$f$\ is continuous on $A$.
\end{remark}

The weak lipschitzianity is a generalization of the Lipschitz condition. We have the following

\begin{proposition}\label{prop: Lip=>weak-Lip}
Let $f: A\longrightarrow \R^{m}$\ be a locally Lipschitz mapping.
Assume that the set $A$\ admits a $C^{q}$ stratification $\sX_{A}$\ such that for all strata $\Gamma\in\sX_{A}$\ the map $f|_{\Gamma}$\
is of class $C^{q}$. Then $f$\ is weakly Lipschitz of class $C^{q}$\ on the stratification $\sX_{A}$.
\end{proposition}

By the $C^{q}$\ Cell Decomposition Theorem (see \cite{DM}), we have the following

\begin{corollary} Let $f: A\longrightarrow\R^{m}$\ be a definable locally Lipschitz mapping.
There exists a definable $C^{q}$\ stratification $\sX_{A}$\ of the set $A$,\ such that $f$\ is weakly Lipschitz of class $C^{q}$\ on $\sX_{A}$.
\end{corollary}

\begin{remark}
Weakly Lipschitz mappings may not be locally Lipschitz (see \cite{Cz} Examples 2.6 and 2.7).
\end{remark}

Proofs of the following propositions are straightforward.

\begin{proposition}
\label{prop: substrat-preserv-WL} Let $A\subset\R^{n}$, $\sX_{A}$\ be a $C^{q}$\ stratification of the set $A$\ and
 $f: A\longrightarrow \R^{n}$\ be weakly Lipschitz of class $C^{q}$\ on $\sX_{A}$. Let $B\subset A$.
Then for any $C^{q}$\ stratification $\sX_{B}$\ of the set $B$, compatible with $\sX_{A}$, the mapping $f|_{B}$\ is weakly Lipschitz of
class $C^{q}$\ on the stratification $\sX_{B}$.
\end{proposition}

\begin{proposition}\label{prop: WL wyst sprawdzac na najwiekszych parach} Let $\Lambda$\ be a $C^{q}$\ submanifold of $\R^{n}$, $\dim\overline{\Lambda}\setminus\Lambda<\dim \Lambda$.
Let $\{\Gamma_{i}\}_{i\in I}$\ be a $C^{q}$\ stratification of $\overline{\Lambda}\setminus\Lambda$.
Let $f: \overline{\Lambda}\longrightarrow \R^{m}$\ be a continuous mapping such that for any $i\in I$\ the restriction $f|_{\Gamma_{i}}$\ is of class $C^{q}$. Then
$f$\ is weakly Lipschitz of class $C^{q}$\ on $\{\Lambda\}\cup\{\Gamma_{i}\}_{i\in I}$\ if and only if for any $i\in I$\ the mapping
 $f$\ is weakly Lipschitz of class $C^{q}$\ on $\{\Lambda,\Gamma_{i}\}$.
\end{proposition}

\begin{proposition}\label{prop: comp-weak-lip-weak-lip} Let $f: A\longrightarrow \R^{p}$\ be a weakly Lipschitz $C^{q}$\ mapping\
on a $C^{q}$\ stratification $\sX_{A}$\ of a set $A\subset\R^{n}$\ and let $g: B\longrightarrow\R^{r}$\ be a weakly Lipschitz $C^{q}$\ mapping on a
$C^{q}$\ stratification $\sX_{B}$\ of a set $B\subset\R^{p}$. Assume that the image under $f$\ of each stratum from $\sX_{A}$\ is contained in some stratum
from $\sX_{B}$\ (in particular, $f(A)\subset B$). Then $g\circ f: A\longrightarrow \R^{r}$\ is a weakly Lipschitz $C^{q}$\ mapping on $\sX_{A}$.
\end{proposition}

\begin{proposition}\label{prop: iloczyn WL jest WL} Let $f: A\longrightarrow \R^{p}$\ be a weakly Lipschitz $C^{q}$\ mapping\
on a $C^{q}$\ stratification $\sX_{A}$\ of a set $A\subset\R^{n}$\ and let $g: B\longrightarrow\R^{r}$\ be a weakly Lipschitz $C^{q}$\ mapping on a
$C^{q}$\ stratification $\sX_{B}$\ of a set $B\subset\R^{m}$. Then $f\times g: A\times B \longrightarrow \R^{p}\times\R^{r}$\ is weakly
Lipschitz of class $C^{q}$\ on a $C^{q}$\ stratification $\sX_{A\times B}=\{\Lambda'\times\Lambda'': \Lambda'\in \sX_{A}, \Lambda''\in\sX_{B}\}$.
\end{proposition}

\begin{definition} Let $A\subset\R^{n}$. For a homeomorphic embedding $f: A\longrightarrow \R^{m}$\ and a $C^{q}$\ stratification $\sX_{A}$\ of $A$\ such that
for any $\Gamma\in\sX_{A}$\ the mapping $f|_{\Gamma}$\ is a $C^{q}$\ embedding, we have a natural $C^{q}$\ stratification of the image $f(A)$\
$$f\sX_{A}=\{f(\Gamma):\quad \Gamma\in\sX_{A}\}.$$
\end{definition}

This leads to the following definition of a weakly bi-Lipschitz homeomorphism:

\begin{definition} Let $A\subset \R^{n}$\ be a set, $f:A\longrightarrow \R^{m}$\ be a homeomorphic embedding. Let
$\sX_{A}$\ be a $C^{q}$\ stratification ($q\geq 1$)\ of $A$\ such that for each $\Gamma\in\sX_{A}$\ the mapping $f|_{\Gamma}$\ is a
$C^{q}$\ embedding.

We say that the mapping $f$\ is \textit{weakly} \textit{bi-Lipschitz of class} $C^{q}$\ \textit{ on the stratification }$\sX_{A}$,\ if $f$\
is weakly Lipschitz of class $C^{q}$\ on $\sX_{A}$\ and its inverse $f^{-1}: f(A)\longrightarrow A\subset\R^{n}$\ is weakly Lipschitz of class $C^{q}$\ on the
stratification $f\sX_{A}$.
\end{definition}

In order to check that the inverse mapping is weakly Lipschitz on $f\sX_{A}$, we can use the following proposition (see \cite{Cz} Prop. 2.14).

\begin{proposition}
\label{prop: weak-lip-condition-for-inverse-mapping}
Let $A\subset \R^{n}$, $f: A\longrightarrow \R^{m}$\ be a homeomorphic embedding. Let $\sX_{A}$\ be a
$C^{q}$\ stratification of $A$\ and assume that for each stratum $\Gamma\in\sX_{A}$\ the mapping $f|_{\Gamma}$\ is
a $C^{q}$\ embedding.
Then the mapping $f^{-1}: f(A)\longrightarrow A$\ is weakly Lipschitz of class $C^{q}$\ on the stratification $f\sX_{A}$,\
if and only if for any strata $\Gamma, \Lambda\in \sX_{A}$\ such that $\Gamma\subset\overline{\Lambda}\setminus\Lambda$\ and for any point $c\in\Gamma$,
if $\{a_{\nu}\}_{\nu\in\N}$, $\{b_\nu\}_{\nu\in\N}$\ are arbitrary sequences such that
$a_{\nu}\in \Gamma, b_{\nu}\in\Lambda$\ for $\nu\in\N$, then
$$a_{\nu}, b_{\nu}\longrightarrow c \quad(\nu\longrightarrow +\infty)\quad \Longrightarrow \quad
\liminf_{\nu\longrightarrow +\infty}\frac{|f(a_{\nu})-f(b_{\nu})|}{|a_{\nu}-b_{\nu}|}> 0. $$
\end{proposition}

\section{Guillaume Valette's Theorem}
\label{sec: Bi-Lipschitz homeo}

Our important tool will be a detailed version of a theorem of Guillaume Valette \cite{Val}, Proposition 3.13. In this section we will use the notation from
\cite{Val}.

 \begin{definition}\label{def: regular direction}
 Let $A$\ be a definable subset of $\R^{n}$\ and let $\lambda\in\R^{n}$\ be a unit vector (i.e. $|\lambda|=1$). We say that $\lambda$\ is \textit{a regular
 direction for} $A$\ if there is a constant $\alpha>0$\ such that for any $C^{1}$\ regular point $a$\ of $A$\ and any unit vector $v$\ tangent to $A$\ at $a$\
 we have $|v-\lambda|\geq \alpha$. Then we say that the orthogonal projection in the direction $\lambda$\ is a \textit{regular projection for} $A$.
 \end{definition}

\begin{theorem}
\label{trm: GV2 bilipsch homeom} Let $A\subset\R^{n}$\ be a definable subset of empty interior. Then there exists a
definable bi-Lipschitz homeomorphism
$\widetilde{h}:\R^{n}\rightarrow \R^{n}$\ such that $\widetilde{h}(A)$\ has a regular projection.
\end{theorem}
\begin{proof}
See \cite{Val}, Proposition 3.13.
\end{proof}

To state a more detailed version, we adopt the following convention. If $\Lambda\subset\R^{n}$\ and $f:\R^{n}\longrightarrow\R$\ is a function, we will identify the restricted function $f|_{\Lambda}$\ with its graph; i.e. $\{(x,y): x\in\Lambda, y=f(x)\}$\ and if $g:\R^{n}\longrightarrow \R$\ is another function, then we put
$$(f,g)|_{\Lambda}=\{(x,y): x\in\Lambda, f(x)<y<g(x)\}.$$

\begin{theorem}\label{trm: biLip homeo could be C1 imbedding}
Let $C$\ be a definable, closed and nowhere dense subset of $\R^{n}$. Let $D_{1},...,D_{s}$\ be definable subsets of\ \ $C$.
There exists $\widetilde{h}: \R^{n}\longrightarrow \R^{n}$\ a definable, bi-Lipschitz homeomorphism such that $e_{n}$\ is a regular
direction for $\widetilde{h}(C)$\ and there exists $\mathcal{C}'$\ a definable $C^{q}$\ stratification of $\R^{n-1}$\
and $\eta_{k}:\R^{n-1}\longrightarrow \R$, $k=1,...,b$\ definable, Lipschitz functions such that $\eta_{1}\leq \eta_{2}\leq ...\leq \eta_{b}$,
$\widetilde{h}(C)\subset\underset{k=1}{\overset{b}{\bigcup}} \eta_{k}$\ and
\vskip 1 mm
\begin{tabular}{l}
 a) for any $\Gamma\in\mathcal{C}'$\ and for any $k=1,...,b$\ the restriction $\eta_{k}|_{\Gamma}$\ is of class $C^{q}$, \\
 b) for any $\Gamma\in\C'$\ and $k=1,...,b-1$,\ $\eta_{k}=\eta_{k+1}$\ on $\Gamma$\ or $\eta_{k}<\eta_{k+1}$\ on $\Gamma$, \\
 c) the family $\mathcal{C}=\left\{(\eta_{k},\eta_{k+1})|_{\Gamma}:\text{  }\Gamma\in \mathcal{C}', \eta_{k}|_{\Gamma}<\eta_{k+1}|_{\Gamma}, k=1,...,b-1\right\}\cup
 \left\{(-\infty,\eta_{1})|_{\Gamma}: \Gamma\in \mathcal{C}'\right\}\cup $\\
 \hskip 5 mm $\cup \left\{(\eta_{b},+\infty)|_{\Gamma}: \Gamma\in \mathcal{C}' \right\}\left\{\eta_{k}|_{\Gamma}: \Gamma\in \mathcal{C}', k=1,...,b\right\}$\ is a definable $C^{q}$\ stratification of $\R^{n}$, compatible \\
 \hskip 5 mm with $\widetilde{h}(C)$, $\widetilde{h}(D_{j})$\ for $j=1,...,s$, \\
 d) for any $\Lambda\in\mathcal{C}$\ the restriction $\widetilde{h}^{-1}|_{\Lambda}$\ is a definable $C^{q}$\ embedding.\\
\end{tabular}
\vskip 1 mm
\end{theorem}

\begin{remark} There is a misprint in the paper \cite{Val} in the proof of Proposition 3.13. It is written there that for $q\in\R^{n-1}$:
$$\eta_{k+1}(q)=\eta_{k}\circ \pi_{e_{n}}(q)+(\xi'_{k}-\xi_{k})\circ \pi_{\lambda_{k}}\circ \widetilde{h}^{-1}(q;\eta_{k}\circ\pi_{e_{n}}(q)).$$
As $\pi_{e_{n}}(q)=q$, then the formula for the mapping $\eta_{k+1}$\ is the following
$$\eta_{k+1}(q)=\eta_{k}(q)+(\xi'_{k}-\xi_{k})\circ \pi_{\lambda_{k}}\circ \widetilde{h}^{-1}(q;\eta_{k}(q)).$$
\end{remark}

\begin{proof} A construction of $\widetilde{h}$\ is the same as in the proof of Proposition 3.13 \cite{Val}.
In addition, observe that for $k=1,...,b$\ we have a definable, bi-Lipschitz homeomorphism $\psi_{0k}: N_{\lambda_{k}}\longrightarrow \R^{n-1}$
$$\psi_{0k}=(id_{\R^{n-1}},\eta_{k})^{-1}\circ \widetilde{h}|_{H_{k}}\circ (id_{N_{\lambda_{k}}}, \xi_{k}),$$
such that for each $z\in\R^{n-1}$

$a)$ $\widetilde{h}^{-1}(z+\eta_{k}(z)\cdot e_{n})=\psi_{0k}^{-1}(z)+\xi_{k}\circ \psi_{0k}^{-1}(z)\cdot \lambda_{k}$

$b)$\ $\widetilde{h}^{-1}(z+\eta_{k}(z)\cdot e_{n}+ \tau \cdot (\eta_{k+1}(z)-\eta_{k}(z))\cdot e_{n})=\psi^{-1}_{0k}(z)+\xi_{k}\circ \psi_{0k}^{-1}(z)\cdot \lambda_{k}+\tau\cdot(\eta_{k+1}(z)-\eta_{k}(z))\cdot \lambda_{k}$

\hskip 13,5 cm for $0<\tau\leq 1$ ;

$c)$\ $\widetilde{h}^{-1}(z+\eta_{b}(z)\cdot e_{n}+\tau\cdot e_{n})=\psi_{0b}^{-1}(z)+\xi_{b}\circ\psi_{0b}^{-1}(z)\cdot\lambda_{b}+\tau\cdot\lambda_{b}$\ \ for $\tau>0$.

$d)$\ $\widetilde{h}^{-1}(z+\eta_{1}(z)\cdot e_{n}+\tau\cdot e_{n})=\psi_{01}^{-1}(z)+\xi_{1}\circ\psi_{01}^{-1}(z)\cdot\lambda_{1}+\tau\cdot\lambda_{1}$\ \ for $\tau\leq 0$.
\vskip 1 mm
 As $\C'$\ it is enough to take a definable $C^{q}$\ stratification of $\R^{n-1}$\ compatible with
  the sets: $\{\eta_{k}=\eta_{k+1}\}$,\ $\{\eta_{k}<\eta_{k+1}\}$, $\pi_{\R^{n-1}}(\widetilde{h}(C)\cap\eta_{k})$, $\pi_{\R^{n-1}}(\widetilde{h}(D_{j})\cap \eta_{k})$\ and such that for each $\Gamma\in\C'$\ the mappings $\eta_{k}|_{\Gamma}$, $\xi_{k}\circ \psi_{0k}^{-1}|_{\Gamma}$\ are of class $C^{q}$\ and $\psi_{0k}^{-1}|_{\Gamma}$\ is a $C^{q}$\ embedding.

\end{proof}

\section{A definable, locally Lipschitz, weakly bi-Lipschitz triangulation}
\label{sec: Weak lipsch triang}

In this section for any definable subset we construct a triangulation that is locally Lipschitz and weakly bi-Lipschitz on the natural simplicial stratification of a simplicial complex. We work in an o-minimal structure on the ordered field $\mathbb{R}$, which admits
$C^{q}$\ Cell Decomposition Theorem, where $q\in\N\cup\{\infty,\omega\}$, $q\geq 1$.
We start with recalling standard definitions concerning triangulations.

\begin{definition}
Let $V$\ be an affine subspace of $\R^{n}$, $\Gamma$\ be a $C^{q}$\ submanifold of $\R^{n}$, $\Gamma\subset V$. Fix a point
$c\in\R^{n}\setminus V$. A \textit{cone with a vertex} $c$\ \textit{and a basis} $\Gamma$\ is a $C^{q}$\ submanifold
$$c*\Gamma=\{(1-t)\cdot c+ t\cdot x:\quad x\in\Gamma,\text{  } t\in (0,1)\}.$$
\end{definition}

\begin{definition} Let $k\in\N$, $k\leq n$. A $k$\ \textit{- dimensional simplex} in $\R^{n}$\ is a set
$$[y_{0},...,y_{k}]=\left\{\underset{i=0}{\overset{k}{\sum}}\beta_{i}\cdot y_{i}:
\quad \beta_{i}>0, i=0,...,k,\text{  } \underset{i=0}{\overset{k}{\sum}}\beta_{i}=1\right\},$$

where $y_{0},...,y_{k}$\ are affinitely independent in $\R^{n}$\ and are called the \textit{vertexes}.
\end{definition}

\begin{remark}\label{rem: simplex - open subset}
Observe that a simplex $\triangle=[y_{0},...,y_{k}]$\ is an open subset in the affine subspace $L$\ spanned by the points $y_{0},...,y_{k}$.
Then $\overline{\triangle}=\left\{\underset{i=0}{\overset{k}{\sum}}\beta_{i}\cdot y_{i}:
\quad\beta_{i}\geq 0, i=0,...,k,\text{  } \underset{i=0}{\overset{k}{\sum}}\beta_{i}=1\right\}$\ is the closure of $\triangle$\ and
$\partial\triangle=\overline{\triangle}\setminus \triangle$\ is the boundary of $\triangle$\ in $L$.
\end{remark}

\begin{definition} Let $l\in\N$, $k\in\N$, $l\leq k$. An $l$\ \textit{- dimensional face} of a simplex
 $\triangle=[y_{0},...,y_{k}]$\ is the following simplex $\triangle'$:
$$\triangle'=[y_{\nu_{0}},...,y_{\nu_{l}}],$$

where $0\leq \nu_{0}<...<\nu_{l}\leq k$.
\end{definition}

\begin{definition}
If $\triangle=[y_{0},...,y_{k}]$\ is a $k$\ - dimensional simplex in $\R^{n}$, then a \textit{barycentre} of $\triangle$\ is the point
$$0_{\triangle}=\overset{k}{\underset{i=0}{\sum}}\frac{1}{k+1}\cdot y_{i}.$$
\end{definition}

\begin{definition} A \textit{simplicial complex} in $\R^{n}$\ is a finite family $K$\ of simplexes in $\R^{n}$,\ which satisfies the following
conditions:

\begin{tabular}{l}
$1)$\ \ for any $S_{1},S_{2}\in K$, $S_{1}\neq S_{2}$,\ we have $S_{1}\cap S_{2}=\emptyset$. \\
$2)$\ \ if $S\in K$\ and $S'$\ is a face of $S$, then $S'\in K$.
\end{tabular}

A \textit{polyhedron of a simplicial complex} $K$\ is the set $$|K|=\bigcup K.$$
Observe that $|K|$\ is a definable compact subset of $\R^{n}$\ of dimension
$\dim K=\max\{\dim \triangle:\text{ }\triangle\in K\}$.
\end{definition}

\begin{definition} For $l\leq n$\ the $l$\ \textit{- dimensional skeleton }of a complex $K$\ is a subcomplex $K^{(l)}$\ of $K$\ defined by:
$$K^{(l)}=\{S\in K:\text{   }\dim S\leq l\}.$$
\end{definition}

\begin{definition}
Let $K$\ be a simplicial complex in $\R^{n}$. Then we define the \textit{barycentric subdivision} $K^{*}$\ of $K$\ by induction on $\dim K$.
\vskip 1 mm
$I. \text{ If }\dim K=0$\ put $K^{*}=K$.

$II. \text{ Assume that }$\ $\dim K=d>0$. By the induction hypotesis we have $\left(K^{(d-1)}\right)^{*}$\ which is a stratification of $K^{(d-1)}$.
Then we put
$$K^{*}= \left(K^{(d-1)}\right)^{*} \cup \left\{0_{\triangle}\ast \triangle':\text{  } \triangle'\subset \partial\triangle,
\triangle'\in \left(K^{(d-1)}\right)^{*},\triangle\in K, \dim\triangle=d\right\}\cup\{0_{\triangle}:\text{  } \triangle\in K,\dim \triangle=d\}.$$
\end{definition}

Now we will prove a lemma, that will be crucial for the proof of the theorem about the existence of a definable, locally Lipschitz,
weakly bi-Lipschitz triangulation.

\begin{lemma}\label{lem: weakly bi-bipschitzianity on polyhedron}
 Let $\widetilde{y}_{0},...,\widetilde{y}_{k}$\ be affinitely independent in $\R^{n-1}$\ and $T=[\widetilde{y}_{0},...,\widetilde{y}_{k}]$\
 be a simplex. Let
$$K_{T}=\{[\widetilde{y}_{\nu_{0}},...,\widetilde{y}_{\nu_{l}}]: 0\leq \nu_{0}<...<\nu_{l}\leq k, l\in\{0,...,k\}\}$$\ be the simplicial complex\ of all faces of $T$, so $\left|K_{T}\right|=\overline{T}$.
Let $f: |\overline{T}|\longrightarrow \R$, $g:|\overline{T}|\longrightarrow \R$\ be definable and Lipschitz functions such that
for each $\triangle\in K_{T}$\

\begin{tabular}{l}
$i)$\ $f= g$\ on $\overline{\triangle}$\ or there is a vertex $w$\ of $\triangle$\ such that $f(w)<g(w)$ \\
$ii)$\ $f|_{\triangle}$, $g|_{\triangle}$\ are of class $C^{q}$.
\end{tabular}\\
Let $\psi_{f}: |\overline{T}|\longrightarrow \R$, $\psi_{g}: |\overline{T}|\longrightarrow\R$\ be semilinear functions defined by formulae:
for any $y=\overset{k}{\underset{i=0}{\sum}}\beta_{i}\widetilde{y}_{i}\in \overline{T}$, where $\overset{k}{\underset{i=0}{\sum}}\beta_{i}=1$, $\beta_{i}\geq 0$, we have
$$\psi_{f}\left(\overset{k}{\underset{i=0}{\sum}}\beta_{i}\widetilde{y}_{i}\right) =\overset{k}{\underset{i=0}{\sum}}\beta_{i}\cdot f(\widetilde{y}_{i}), \qquad
\psi_{g}\left(\overset{k}{\underset{i=0}{\sum}}\beta_{i}\widetilde{y}_{i}\right) =\overset{k}{\underset{i=0}{\sum}}\beta_{i}\cdot g(\widetilde{y}_{i}).$$
Consider the following polyhedral complex
$$K_{p}=\{\psi_{f}|_{\triangle}: \triangle\in K_{T}\}\cup\{\psi_{g}|_{\triangle}: \triangle\in K_{T}\}\cup\{\left(\psi_{f}, \psi_{g}\right)|_{\triangle}:
\triangle\in K_{T}, \psi_{f}|_{\triangle}<\psi_{g}|_{\triangle}\}$$
and put $\left|K_{p}\right|=\bigcup K_{p}$. Then the mapping $H: \left|K_{p}\right|\longrightarrow \R^{n}$\ defined by the following formula:
$$ (\mathfrak{F})\quad H(y,z)=\begin{cases}(y,f(y)), & (y,z)\in \psi_{f}|_{\triangle}, \triangle \in K_{T}, \\
\left( y,\frac{z-\psi_{f}(y)}{\psi_{g}(y)-\psi_{f}(y)}\cdot g(y)+
 \frac{\psi_{g}(y)-z}{\psi_{g}(y)-\psi_{f}(y)}\cdot f(y) \right), & (y,z)\in (\psi_{f}, \psi_{g})|_{\triangle}, \triangle\in K_{T},\\
(y,g(y)), & (y,z)\in \psi_{g}|_{\triangle}, \triangle \in K_{T} \\ \end{cases}$$
 is a definable homeomorphic embedding such that
 $$H\left( \left| K_{p}\right| \right) =\left\{ (y,z)\in\overline{T}\times\R: f(y)\leq z\leq g(y)\right\} $$
 and
 $$\{H(S):\text{\ \ }S\in K_{p}\}=\{f|_{\triangle}:\text{ }\triangle\in K_{T}\}\cup\{g|_{\triangle}:\text{ }\triangle\in K_{T}\}
 \cup\{(f,g)|_{\triangle}: \text{ }\triangle\in K_{T}, f|_{\triangle}<g|_{\triangle}\}$$\
 and moreover \\
 \begin{tabular}{l}
 a) $H$\ is a Lipschitz mapping; \\
 b) $H^{-1}$\ is locally Lipschitz on $\{(y,z)\in \overline{T}\times\R: f(y)\leq z\leq g(y), f(y)<g(y)\}$;\\
 c) $H$\ is weakly bi-Lipschitz of class $C^{q}$\ on $K_{p}$.
 \end{tabular}
\end{lemma}
\begin{proof}
We will prove $a)$, $b)$\ and $c)$\ leaving the other - more standard part - to the reader. Notice that $H$\ is well-defined due to $i)$\ which implies
that for each $\triangle\in K_{T}$\ $f<g$\ on $\triangle$\ if and only if $\psi_{f}<\psi_{g}$\ on $\triangle$\ and $f\equiv g$\ on $\triangle$\ if and only if
$\psi_{f}\equiv \psi_{g}$\ on $\triangle$.

To prove $a)$\ first observe that using the following Lipschitz automorphism
$$\chi: \overline{T}\times \R\ni (y,z)\longmapsto (y, z-f(y))\in \overline{T}\times \R$$
we can assume without loss of generality that $f\equiv \psi_{f}\equiv 0$. Clearly, we can also assume that $g>0$\ on $T$. Put $S=(0,\psi_{g})|_{T}$\ and
$H(y,z)=(y, H_{2}(y,z))$. Since $S$\ is a convex set and $H|_{\overline{S}}$\ is continuous\ and $H|_{S}$\ is of class $C^{q}$, to prove that $H$\ is Lipschitz it
suffices\footnote{Thanks to the Mean Value Theorem, see \cite{Di} Theorem 8.5.2.}\ to check that all first-order partial derivatives of $H$\ are bounded on $S$.
Since
$$\frac{\partial H_{2}}{\partial{y_{i}}}(y,z)=\frac{z}{\psi_{g}(y)}\cdot\frac{\partial g(y)}{\partial y_{i}}-\frac{z}{\psi_{g}(y)}\cdot\frac{g(y)}{\psi_{g}(y)}\cdot
\frac{\partial\psi_{g}(y)}{\partial y_{i}}\qquad \text{and}\qquad \frac{\partial H_{2}}{\partial z}(y,z)=\frac{g(y)}{\psi_{g}(y)},$$\
it is enough to check that $\frac{g}{\psi_{g}}$\ is bounded on $T$. This is clear if $\psi_{g}(\widetilde{y}_{j})=g(\widetilde{y}_{j})>0$\ for all $j$, so assume that
$\{\widetilde{y}_{0},...,\widetilde{y}_{l}\}=\{\widetilde{y}_{j}: g(\widetilde{y}_{j})=0\}$, where $0\leq l<k$.

For any $y=\underset{\nu=0}{\overset{k}{\sum}}\beta_{\nu}\widetilde{y}_{\nu}\in T$, where $\beta_{\nu}>0$\ and
$\underset{\nu=0}{\overset{k}{\sum}}\beta_{\nu}=1$\ take
$$ x=\frac{\beta_{0}}{\underset{\nu=0}{\overset{l}{\sum}}\beta_{\nu}}\cdot \widetilde{y}_{0}+...+
\frac{\beta_{l}}{\underset{\nu=0}{\overset{l}{\sum}}\beta_{\nu}}\cdot \widetilde{y}_{l}.$$
Then $g(x)=0$\ by $i)$\ and
$$x-y= \beta_{l+1}(x-\widetilde{y}_{l+1})+...+\beta_{k}(x-\widetilde{y}_{k}).$$
Consequently,
$$\left|\frac{g(y)}{\psi_{g}(y)} \right|=
\left|\frac{g(y)-g(x)}{\psi_{g}(y)} \right|\leq \frac{L_{g}\cdot|x-y|}{\psi_{g}(y)}\leq
L_{g}\cdot\frac{\underset{\nu=l+1}{\overset{k}{\sum}}\beta_{\nu}\cdot |x-\widetilde{y}_{\nu}|}
{\underset{\nu=l+1}{\overset{k}{\sum}}\beta_{\nu}\cdot g(\widetilde{y}_{\nu})}\leq
\frac{L_{g}\cdot diam T\cdot (k-l)}{\min\{g(\widetilde{y}_{\nu}): \nu>l\}}.$$
This completes the proof of $a)$.

To prove $b)$\ take any point $(c,d)\in\overline{T}\times\R$\ such that $f(c)\leq d <g(c)$\ (for the case $f(c)<d\leq g(c)$\ the argument will be similar).
Again, using the Lipschitz automorphism $\chi$\ we can assume without loss of generality that $f\equiv \psi_{f}\equiv 0$.
Put
$$\Pi=\{(y,z)\in T\times\R: 0<z<g(y)\}.$$
Then $(c,d)\in\overline{\Pi}$\ and it is clear that $(c,d)$\ admits arbitrarily small neighbourhood $U$\ in $\R^{n}$\ such that $U\cap\Pi$\ is convex.
Hence now it is enough to notice that all first-order partial derivatives of
$$H^{-1}(y,z)=\left(y,\frac{\psi_{g}(y)}{g(y)}\cdot z\right)$$
are bounded on some $U\cap\Pi$. This completes the proof of $b)$.

To prove $c)$\ take any pair $\Gamma$, $\Lambda\in K_{p}$\ such that $\Gamma\subset\overline{\Lambda}\setminus\Lambda$. We will check that $H$\ is weakly bi-Lipschitz
of class $C^{q}$\ on $\{\Gamma, \Lambda\}$. This is clear in case $\Lambda=\psi_{f}|_{\triangle}$\ or $\Lambda=\psi_{g}|_{\triangle}$, since then
$H|_{\overline{\Lambda}}$\ is bi-Lipschitz. Hence without loss of generality we can assume\footnote{see Proposition \ref{prop: WL wyst sprawdzac na najwiekszych parach}} that $\Lambda=(\psi_{f},\psi_{g})|_{T}$. In view of $a)$\ and $b)$\ we can
also assume that $\Gamma=\psi_{f}|_{\triangle}=\psi_{g}|_{\triangle}$\ for some $\triangle\in K_{T}$, $\triangle\subset \overline{T}\setminus T$.
Take any $c\in\Gamma$\ and two sequences $a_{\nu}\in\Gamma$, $b_{\nu}\in\Lambda$\ ($\nu\in\N$)\ such that $a_{\nu}$, $b_{\nu}\longrightarrow c$. Put
$a_{\nu}=(a_{\nu}', a_{\nu n})$, $b_{\nu}=(b_{\nu}', b_{\nu n})$, where $a_{\nu}'$, $b_{\nu}'\in \R^{n-1}$, $a_{\nu n}, b_{\nu n}\in \R$. Since
$$\underset{\nu\to+\infty}{\liminf}\frac{|H(a_{\nu})-H(b_{\nu})|}{|a_{\nu}-b_{\nu}|}\geq
\underset{\nu\to+\infty}{\liminf}\frac{|a_{\nu}'-b_{\nu}'|}{|a_{\nu}-b_{\nu}|}>0,$$
the proof is completed by Proposition \ref{prop: weak-lip-condition-for-inverse-mapping}.
\end{proof}

Remind now briefly a definition of a definable $C^{q}$\ triangulation, $q\in\N\cup\{\infty,\omega\}$, $q\geq 1$.

\begin{definition}
Let $A$\ be a compact definable set in $\R^{n}$, $q\in\N\cup\{\infty, \omega\}$.
A \textit{definable} $C^{q}$ \textit{triangulation} of the set $A$\ is a pair $(K,f)$\ of a simplicial complex $K$\ and
a definable homeomorphism $f: |K|\longrightarrow A $ \ such that for each $\triangle \in K$\
the set $f(\triangle)$\ is a definable $C^{q}$\ submanifold of $\R^{n}$\ and
$f|_{\triangle}$\ is a definable $C^{q}$\ diffeomorphism onto $f(\triangle)$.
If $A$\ is not compact, then a \textit{definable $C^{q}$\ triangulation }of $A$\ is a pair $(K',f)$, where
$K'$\ is a subfamily of a simplicial complex $K$\ and
$f: |K'|\longrightarrow A $\ is a definable homeomorphism of $|K'|=\bigcup K'$\ onto $A$\ such that for each $\triangle \in K'$\
the set $f(\triangle)$\ is a definable $C^{q}$\ submanifold of $\R^{n}$\ and
$f|_{\triangle}: \triangle\longrightarrow f(\triangle)$\ is a definable $C^{q}$\ diffeomorphism.
\vskip 1 mm
Let $A_{1},...,A_{r}$\ be definable subsets of $A$.
We say that a triangulation $(K,f)$\ is \textit{compatible with the sets} $A_{1},...,A_{r}$, if
the stratification $\{f(\triangle):\quad\triangle\in K\}$\ is compatible with $A_{1},..., A_{r}$.
\end{definition}

We may prove now one of the main theorems in this paper.

\begin{theorem}\label{trm: Weakly_lipschitz_triang}(Definable, locally Lipschitz, weakly bi-Lipschitz triangulation) Let $A$\ be a definable subset
of $\R^{n}$, $A_{1},..., A_{r}$\ be definable subsets of $A$, $r\in\N$.
Then there exists a definable $C^{q}$\ triangulation $(K,H)$\ of the set $A$, compatible with $A_{1}, ..., A_{r}$\ and such that
\vskip 1 mm

\begin{tabular}{l}
 a) \ $H$\ \ is a locally Lipschitz mapping; \\
 b) \ $H$\ is weakly bi-Lipschitz of class $C^{q}$\ on the natural simplicial stratification $K$\
 of the set $|K|$.\\
\end{tabular}
\end{theorem}
\begin{proof} By the following definable $C^{\omega}$\ diffeomorphism
$$\zeta: \R^{n}\ni x\longmapsto \frac{x}{\sqrt{1+|x|^{2}}}\in \{u\in\R^{n}: |u|<1\}$$
without loss of generality we can assume that $A$\ is compact.
Observe also that it suffices to find a definable $C^{q}$\ triangulation $(K,H)$\ of the set $A$, compatible with $A_{1},...,A_{r}$\ and such that for any $\triangle\in K$
\vskip 1 mm

\begin{tabular}{l}
 a) \ $H|_{\overline{\triangle}}$\ \ is a Lipschitz mapping\footnotemark; \\
 b) \ $H|_{\overline{\triangle}}$\ is weakly bi-Lipschitz of class $C^{q}$\ on the natural simplicial stratification $\sX_{\overline{\triangle}}$\
  of the set $\overline{\triangle}$.\\
\end{tabular}

\footnotetext{It suffices to apply the following lemmas about "glueing" Lipschitz mappings: Let $A_{i}\subset\R^{p}$, $i=1,...,r$\ and  $B=\overset{r}{\underset{i=1}{\bigcup}} A_{i}$\ be a quasi-convex\footnotemark\ set, $f: B\longrightarrow \R^{m}$\ be such that for $i=1,..., r$,
$f|_{A_{i}}$\ is a Lipschitz mapping. Then $f$\ is a Lipschitz mapping. Also if $G_{1},...,G_{r}\subset\R^{n}$\ are disjoint and compact, $f: \overset{r}{\underset{i=1}{\bigcup}} G_{i}\longrightarrow \R^{m}$\ is such that $f|_{G_{i}}$\ is Lipschitz, then $f$\ is Lipschitz.}
\footnotetext{A set $A\subset\R^{n}$\ is \textit{quasi-convex}, if there exists a constant $C>0$\ such that for $x,y\in A$\ there exists
a continuous arc $\lambda: [0,1]\longrightarrow A$,\ piecewise $C^{1}$\ such that $\lambda(0)=x$, $\lambda(1)=y$\ and
$|\lambda|\leq C\cdot |x-y|$. In particular the connected components of $|K|$, where $K$\ is a simplicial complex in $\R^{n}$, are quasi-convex.}

The proof is proceeded by induction on the dimension $n$\ of the ambient space.

Cases $n=0$, $n=1$\ are trivial.

Let $n\geq 2$\ and the statement be true for $n-1$.
Denote $D_{0}=A\setminus int A$, $D_{i}= A_{i}\setminus int A_{i}$, $D_{i+r}= \overline{A}_{i}\setminus int A_{i}$, $i=1,2,...,r$.
Let $C=\underset{i=0}{\overset{2r}{\bigcup}}D_{i}$.
Observe that $C$\ is nowhere dense in $\R^{n}$\ and that any stratification of $\R^{n}$, compatible with all the sets
$D_{0},...,D_{2r}$\ is also compatible with all the sets $A,A_{1},...,A_{r}$.
\vskip 3 mm

\textbf{Step 1.} By Theorem \ref{trm: biLip homeo could be C1 imbedding} we can assume that there exists a definable $C^{q}$\ stratification $\C'$\ of
$\R^{n-1}$\ and a finite family of definable Lipschitz functions $\eta_{k}: \R^{n-1}\longrightarrow \R$\ ($k=1,...,b$)\ such that
 $\eta_{1}\leq...\leq \eta_{b}$, $C\subset\underset{k=1}{\overset{b}{\bigcup}} \eta_{k}$\ and
for any $\Gamma\in \C'$, $\eta_{k}|_{\Gamma}$\ is of class $C^{q}$\ and either $\eta_{k}|_{\Gamma}\equiv \eta_{k+1}|_{\Gamma}$\ or
$\eta_{k}|_{\Gamma}<\eta_{k+1}|_{\Gamma}$\ and the family
{\setlength\arraycolsep{2pt}
\begin{eqnarray*}
\mathcal{C}&:=&\left\{(\eta_{k},\eta_{k+1})|_{\Gamma}: \Gamma\in \mathcal{C}', \eta_{k}|_{\Gamma}<\eta_{k+1}|_{\Gamma}, k=1,...,b-1\right\}\cup
\left\{(\eta_{b},+\infty)|_{\Gamma}: \Gamma\in \mathcal{C}' \right\}\nonumber \\
& & \cup \left\{(-\infty,\eta_{1})|_{\Gamma}: \Gamma\in \mathcal{C}'\right\}\cup \left\{\eta_{k}|_{\Gamma}: \Gamma\in \mathcal{C}', k=1,...,b\right\}
\end{eqnarray*}}
is a $C^{q}$\ stratification of $\R^{n}$, compatible with $A, A_{1},...,A_{r}$. Now it suffices to construct a desired triangulation compatible with
$\C$\ for the set
$$\{(y,z)\in A'\times \R: \eta_{1}(y)\leq z\leq \eta_{b}(y)\},$$
where $A'$\ is the projection of $A$\ into $\R^{n-1}$.
\vskip 3 mm

\textbf{Step 2.} From now on we will follow the classical construction of triangulation (compare \cite{Hi2} or \cite{Co} Theorem 4.4).
By the induction hypothesis there exists a definable $C^{q}$\ triangulation $(K',h)$\ of the set
$A'$, that is compatible with the finite family $\{\Gamma\in\C':\quad \Gamma \subset A'\}$\
and such that for every simplex $\triangle\in K'$
\vskip 1 mm
\begin{tabular}{l}
 $a)$\ \ $h|_{\overline{\triangle}}$\ \  is a Lipschitz mapping; \\
 $b)$\ \ $h|_{\overline{\triangle}}$\ \ is weakly bi-Lipschitz on the natural simplicial stratification $\sX_{\overline{\triangle}}$\ \
 of the set $\overline{\triangle}$,\\
 \hskip 14 mm where $\sX_{\overline{\triangle}}=\{\triangle'\in K':\text{ }  \triangle'\subset \partial\triangle \}\cup \{\triangle\}$. \\
\end{tabular}\\
Without loss of generality, by taking the barycentric subdivision of $K'$, we can assume
(see Proposition \nolinebreak\ref{prop: substrat-preserv-WL}) that $(K', h)$\ has still the above properties $a)$\ and $b)$\ and furthermore
$$\eta_{k}\circ h|_{\overline{\triangle}} = \eta_{k+1}\circ h|_{\overline{\triangle}}\quad\text{or there is a vertex } w \text{ of }\triangle
\text{ such that }\eta_{k}\circ h(w)<\eta_{k+1}\circ h(w)$$
for any $\triangle\in K'$\ and $k=1,...,b-1$.
Thus since now without loss of generality we can assume\footnote{see Proposition\ref{prop: iloczyn WL jest WL}} that $A'=|K'|$, $h=id_{A'}$\ and
$$(\sharp)\qquad\qquad \eta_{k}|_{\overline{\triangle}}=\eta_{k+1}|_{\triangle}\quad
\text{or there is a vertex }w\text{ of }\triangle\text{ such that }\eta_{k}(w)<\eta_{k+1}(w)\qquad\qquad$$
for any $\triangle\in K'$\ and $k=1,...,b-1$.
Consider the following functions $\psi_{k}: |K'|\longrightarrow \R$:
$$\psi_{k}\left(\overset{l}{\underset{i=0}{\sum}}\beta_{i}y_{i}\right) =\overset{l}{\underset{i=0}{\sum}}\beta_{i}\cdot \eta_{k}(y_{i}),
\text{  for any  } [y_{0},...,y_{l}]\in K'\ \text{ and } \beta_{i}>0\text{ such that } \overset{l}{\underset{i=0}{\sum}}\beta_{i}=1.$$
Observe that every $\psi_{k}$\ is affine on the closure of each simplex $\triangle\in K'$\ and by $(\sharp)$\
\begin{center}$\psi_{k}<\psi_{k+1}$\ \ on  $\triangle$\ \ $\Longleftrightarrow$\ \ $\eta_{k}< \eta_{k+1}$\ \ on $\triangle$\ \ \ \ and\ \ \ \
$\psi_{k}=\psi_{k+1}$\ \ on  $\overline{\triangle}$\ \ $\Longleftrightarrow$\ \ $\eta_{k}= \eta_{k+1}$\ \ on $\overline{\triangle}$.
\end{center}
Consider the following polyhedral complex:
$$K_{p}=\{\psi_{k}|_{\triangle}, \triangle\in K', k=1,...,b\}\cup \{(\psi_{k}, \psi_{k+1})|_{\triangle}, \triangle\in K', \psi_{k}|_{\triangle}<\psi_{k+1}|_{\triangle},
k=1,...,b-1\}.$$
Put $\left|K_{p}\right|=\bigcup K_{p}$. Now we define $H: |K_{p}|\longrightarrow \{(y,z)\in A'\times\R: \eta_{1}(y)\leq z\leq \eta_{b}(y)\}$\
by the formula:
$$(\mathfrak{F})\qquad
H(y,z)= \begin{cases}(y,\eta_{k}(y)), & (y,z)\in \psi_{k}|_{\triangle}, \triangle \in K', \\
 & k=1,...,b, \\
\left( y,\frac{z-\psi_{k}(y)}{\psi_{k+1}(y)-\psi_{k}(y)}\cdot \eta_{k+1}(y)+
 \frac{\psi_{k+1}(y)-z}{\psi_{k+1}(y)-\psi_{k}(y)}\cdot \eta_{k}(y) \right), & (y,z)\in (\psi_{k}, \psi_{k+1})|_{\triangle},
  \triangle\in K', \\
  & k=1,...,b-1, \psi_{k}|_{\triangle}<\psi_{k+1}|_{\triangle}. \end{cases}$$
By Lemma \ref{lem: weakly bi-bipschitzianity on polyhedron} $H$\ is a homeomorphism, $\{H(S):\text{\ \ }S\in K_{p}\}$\ is compatible with $\C$\ and
for any $S\in K_{p}$\

\begin{tabular}{l l}
$a)$\ \ $H|_{\overline{S}}$\ \ is a Lipschitz mapping; \\
$b)$\ \ $H|_{\overline{S}}$\ \ is weakly bi-Lipschitz of class $C^{q}$\ on the natural polyhedral stratification $\sX_{\overline{S}}$\
of the set $\overline{S}$, \\ where $\sX_{\overline{S}}=\{S'\in K_{p}:\text{ }S'\subset\overline{S}\}$.
\end{tabular} \\
Finally, passing to the barycentric subdivision $K^{*}_{p}$\ of $K_{p}$\ completes the construction
 (see Proposition \ref{prop: substrat-preserv-WL}).

\end{proof}

\begin{remark} It follows that $K$\ can be always (a subfamily of) a simplicial complex in $\R^{n}$\ for a set $A\subset\R^{n}$.
\end{remark}

\begin{remark}\label{rem: lipschitz triangulation for compact sets} If $A$\ is compact, we can have $H: |K|\longrightarrow A$\ Lipschitz.
\vskip 1 mm
\end{remark}

\section{A class of triangulable regularity conditions}
\label{sec: T class and T triangulation}

Using the results from \cite{Cz}, we define a class $\T$\ of regularity conditions and prove that every definable compact set in $\R^{n}$\
has a definable $C^{q}$\ triangulation with a $\T$\ condition. We remind briefly the important notions from \cite{Cz}.

\vskip 3 mm
Let $\Q$\ be a regularity condition of pairs $(\Lambda,\Gamma)$\ at a point $x\in\Gamma$, where $\Lambda,\Gamma\subset\R^{n}$\
are $C^{q}$\ submanifolds, $\Gamma\subset\overline{\Lambda}\setminus\Lambda$.
\vskip 3 mm

\begin{definition} We say that $\mathcal{Q}$\ is \textit{local}, if for an open neighbourhood $U$\ of the point $x\in\Gamma$
\ the pair $(\Lambda, \Gamma)$\ satisfies the condition $\Q$\ at $x$\ if and only if the pair $(\Lambda\cap U,\Gamma\cap U)$\ satisfies the condition
$\Q$\ at the point $x$.
\end{definition}
Since now we consider only these conditions, that are local. We also set the following notation:
\begin{center}
\begin{tabular}{l}
$\W^{\mathcal{Q}}(\Lambda,\Gamma,x)=$\ the condition $\Q$\ is satisfied for the pair $(\Lambda,\Gamma)$\ at the point $x\in\Gamma$.\\
$\W^{\Q}(\Lambda,\Gamma)=$\ for any point $x\in\Gamma$\ we have $\W^{\mathcal{Q}}(\Lambda,\Gamma,x)$. \\
$\sim\W^{\Q}(\Lambda, \Gamma,x)=$\ the pair $(\Lambda,\Gamma)$\ does not satisfy $\Q$\ at the point $x$.\\
\end{tabular}
\end{center}
In the natural way we may define a stratification with the $\Q$\ condition.
\begin{definition}\label{def: Q stratification} Let $\Q$\ be a regularity condition, $A\subset\R^{n}$\ be a set.
A $C^{q}$\ \textit{stratification with the}\ \textit{condition}\ $\Q$\ (or a $\Q$\ \textit{stratification of class}
$C^{q}$) of the set $A$\ is a $C^{q}$\ stratification $\sX^{\Q}_{A}$\ such that for any two strata $\Lambda$,
$\Gamma\in\sX^{\Q}_{A}$, $\Gamma\subset\overline{\Lambda}\setminus\Lambda$,\ we have $\W^{Q}(\Lambda,\Gamma)$.
\end{definition}
In the next part of this section we focus on describing common features of regularity conditions. Some of them are well-known features of Whitney's conditions but
some seem to be new (the lifting property, the projection property, the conical property).

\begin{definition}\label{def: defin cond}
Let $\Q$\ be a regularity condition. We say that $\Q$\ is \textit{definable}, if for any definable
$C^{q}$\ submanifolds $\Gamma,\Lambda\subset\R^{n}$, $\Gamma\subset\overline{\Lambda}\setminus\Lambda$, the set
$$\{x\in\Gamma:\quad \W^{\Q}(\Lambda,\Gamma,x)\} $$
is definable.
\end{definition}

\begin{definition}\label{def: generic cond}
Let $\mathcal{Q}$\ be a definable regularity condition. We say that $\mathcal{Q}$\ is \textit{generic}, if for any definable
$C^{q}$\ submanifolds $\Lambda,\Gamma\subset\R^{n}$, $\Gamma\subset\overline{\Lambda}\setminus\Lambda$, the set
$$\{x\in\Gamma:\quad \sim\W^{\Q}(\Lambda,\Gamma,x)\}$$
is nowhere dense in $\Gamma$.
\end{definition}

\begin{remark}\label{rem: generic for 0 dimensional} If $\Q$\ is a definable and generic condition, then for any two definable $C^{q}$\ submanifolds
$\Lambda$, $\Gamma\subset \R^{n}$, such that $\Gamma\subset\overline{\Lambda}\setminus\Lambda$\ and $\dim \Gamma=0$, we have always $\W^{q}(\Lambda,\Gamma)$.
\end{remark}

\begin{definition}\label{def: invariant condition Q} Let $\Q$\ be a regularity condition. We say that $\Q$\ is $C^{q}$\ \textit{invariant} (or \textit{invariant under
} $C^{q}$\ \textit{diffeomorphisms}), if for any $C^{q}$\ submanifolds $\Lambda$, $\Gamma\subset\R^{n}$, $\Gamma\subset\overline{\Lambda}\setminus\Lambda$,\ $\dim\Gamma<\dim\Lambda$
and for any point $x\in\Gamma$, if $U$\ is an open neighbourhood of $x$\ and
$\phi: U \longrightarrow \R^{m}$\ is a $C^{q}$\ embedding, then
$$\W^{Q}(\Lambda,\Gamma,x)\quad \Longleftrightarrow \quad \W^{\Q}(\phi(\Lambda\cap U), \phi(\Gamma\cap U),\phi(x)).$$
\end{definition}

\begin{definition}
Let $A\subset \R^{n}$\ and let $f: A\longrightarrow \R^{m}$\ be a continuous mapping, $\sX_{A}$\ be a $C^{q}$\ stratification of the
set $A$\ such that $f|_{\Gamma}$\ is of class $C^{q}$\ for all $\Gamma\in\sX_{A}$.
Then by the \emph{induced} $C^{q}$\ \emph{stratification} of the $graph f$, we will mean the following:
$$\sX_{graph f}(\sX_{A})=\{graph f|_{\Gamma}:\quad \Gamma\in\sX_{A}\}.$$
\end{definition}

\begin{definition}\label{def: property proj weakly lipsch}
We say that a condition $\Q$\ has \textit{the projection property with respect to weakly Lipschitz mappings of class }$C^{q}$\
if for any $C^{q}$\ mapping $f: A\longrightarrow \R^{m}$\ weakly Lipschitz on a $C^{q}$\ stratification $\sX_{A}$\ of a set $A\subset\R^{n}$, we have
$$\sX_{graph\text{ }f}\left(\sX_{A}\right)\text{\ is a\ }\Q\text{\ -\ stratification}\quad\Longrightarrow \quad \sX_{A}\text{\ is a\ }\Q\text{\ -\ stratification}.$$
\end{definition}

\begin{definition}\label{def: property lifting by Lipschitz}
We say that a condition $\Q$\ has \textit{the lifting property with respect to locally Lipschitz mappings of class }$C^{q}$\ if
for any two $C^{q}$\ submanifolds $\Lambda,\Gamma\subset\R^{n}$\ such that $\Gamma\subset\overline{\Lambda}\setminus\Lambda$,
and for any locally Lipschitz mapping $f: \Lambda\cup\Gamma\longrightarrow \R^{m}$\ such that the restrictions
$f|_{\Lambda}$, $f|_{\Gamma}$\ are of class $C^{q}$\ and for any $C^{q}$\ submanifolds $M,N\subset\R^{n}$\ such that
$M,N\subset\Lambda\cup\Gamma$, $N\subset\overline{M}\setminus M$\ and $\{M,N\}$\ is
  compatible with $\{\Lambda,\Gamma\}$, we have
  $$\W^{\Q}(M,N), \W^{\Q}(graph f|_{\Lambda},graph f|_{\Gamma}) \quad \Longrightarrow \quad \W^{\Q}(graph f|_{M},graph f|_{N}).$$
\end{definition}

\begin{definition}\label{def: WL condition} We say that a regularity condition $\Q$\ is a $\mathcal{WL}$\ \textit{condition}\
\textit{of class}\ $C^{q}$, if and only if it has all the following properties: \\
- the definability ; \\
- the genericity ; \\
- the invariance under definable $C^{q}$\ diffeomorphisms ;\\
- the projection property with respect to weakly Lipschitz mappings of class $C^{q}$; \\
- the lifting property with respect to locally Lipschitz mappings of class $C^{q}$.
\end{definition}

The class $\mathcal{WL}$\ and properties of weakly Lipschitz mappings were widely discussed in \cite{Cz}.
The main theorem of \cite{Cz} states that the $\WL$\ conditions are invariant under weakly Lipschitz mappings in the following sense:

\begin{theorem}(Invariance of the $\WL$\ conditions under definable, locally Lipschitz, weakly bi-Lipschitz homeomorphisms)
\label{trm: Q invariance}\hskip 1 mm
Let $\Q$\ be a regularity $\mathcal{WL}$\ condition of class $C^{q}$, where $q\in\N\cup\{\infty,\omega\}$\ and $q\geq 1$.
Let $B\subset \R^{n}$\ be a definable set and consider
a definable homeomorphic embedding $f: B\longrightarrow \R^{m}$, that is weakly bi-Lipschitz of class $C^{q}$\ on
definable $C^{q}$\ stratification $\sX_{B}$. Assume additionally that for any two sumbanifolds $\Lambda$, $\Gamma\in \sX_{B}$\
such that $\Gamma\subset\overline{\Lambda}\setminus\Lambda$,\ the mapping $f|_{\Lambda\cup\Gamma}$\ is locally Lipschitz.

Then there exists  a definable $C^{q}$\ stratification $\sX_{B}'$\ of the set $B$, compatible with $\sX_{B}$\ and such that
 $$\{\Gamma\in\sX_{B}:\quad \dim\Gamma=\dim B\}=\{\Gamma'\in\sX_{B}':\quad \dim\Gamma'=\dim B\}$$
and the condition $\Q$\ is invariant with respect to the pair $(f,\sX_{B}')$\ in the following sense
\vskip 2 mm
 for any definable $C^{q}$\ submanifolds $M,N\subset B$\ such that $N\subset\overline{M}\setminus M$\
and $\{M,N\}$\ are compatible with the stratification $\sX_{B}'$\
$$\W^{\Q}(M,N)\quad\Longrightarrow\quad \W^{\Q}(f(M),f(N)).$$
\end{theorem}
\begin{proof} See \cite{Cz}, Theorem 3.15.
\end{proof}

However, in order to prove the triangulation theorem, we have to narrow the class $\WL$\ of regularity conditions by imposing an extra condition.

\begin{definition}\label{def: conical property} Let $\Q$\ be a regularity condition. We say that $\Q$\ has
 \textit{the conical property of class } $C^{q}$, if for any $n\in\N$, any affine subspace $V\subset\R^{n}$\ and any $C^{q}$\ submanifolds $M$, $N\subset V$\ such that
$N\subset\overline{M}\setminus M$\ and for any point $c\in\R^{n}\setminus V$\ we have
$$\W^{\Q}(M, N)\quad \Longrightarrow \quad \begin{tabular}{l}a) $\W^{\Q}\left(c\ast M,M\right)$, $\W^{\Q}\left(c\ast N,N\right)$; \\
 b) $\W^{\Q}\left(c\ast M, c\ast N\right)$ ; \\ c) $W^{\Q}\left(c\ast M, N\right)$. \\ \end{tabular}$$
\end{definition}

\begin{remark}\label{rem: conical property 1} Observe that for any affine subspace $S\subset\R^{n}$\ and for any point $c\in\R^{n}\setminus S$\ the
mapping
$$\varphi: S\times (0;+\infty)\ni (x,t)\longmapsto (1-t)\cdot c + t\cdot x \in\R^{n}$$
is a $C^{q}$\ embedding for each $q\in\N\cup\{\infty, \omega\}$, $q\geq 1$.
Therefore, if a condition $\Q$\ is invariant under
definable $C^{q}$\ diffeomorphisms (e.g. when $\Q$\ is a $\WL$\ condition of class $C^{q}$), then it has the conical property of class $C^{q}$\ if and only if
for any $C^{q}$\ submanifolds $M$, $N\subset \R^{n}$\ such that $N\subset\overline{M}\setminus M$\ we have
$$\W^{\Q}(M, N)\quad \Longrightarrow \quad \begin{tabular}{l}a) $\W^{\Q}\left(M\times (0;1),M\times\{1\}\right)$, $\W^{\Q}(N\times(0;1),N\times\{1\})$; \\
 b) $\W^{\Q}(M\times(0;1), N\times(0;1))$ ; \\ c) $W^{\Q}( M\times(0;1), N\times \{1\})$. \\ \end{tabular}$$
\end{remark}

Now we can describe a class of triangulable conditions.

\begin{definition} A regularity condition $\Q$\ is called a \textit{triangulable}\ $C^{q}$\ \textit{condition}, if it is a $\WL$\ condition of
class $C^{q}$\ and it has the conical property of class $C^{q}$. Let $\T$\ denotes the class of triangulable conditions.
\end{definition}

In Section \ref{sec: Whitney Verdier are T} we will prove that the Whitney (B) and the Verdier conditions belong to the class $\T$.
Now we will prove a general theorem about definable $C^{q}$\ triangulation with a triangulable condition.

\begin{theorem}\label{trm: triangulation with Q}(Definable, locally Lipschitz triangulation with a triangulable condition)
Let $\Q$\ be a trian\-gu\-lab\-le condition of class $C^{q}$, where $q\in\N\cup\{\infty,\omega\}$\ and $q\geq 1$. Let $A\subset \R^{n}$\ be
a definable set and $A_{1},...,A_{r}$\ be definable subsets of $A$.
Then there exists a definable $C^{q}$\ triangulation $(K,H)$\ of $A$,\ such that the family $\{H(\triangle):\text{   }\triangle\in K\}$\
is a definable $C^{q}$\ stratification with the condition $\Q$\ of the set $A$\ compatible
with $A_{1},...,A_{r}$\ and $H: |K|\longrightarrow A$\ is a locally Lipschitz mapping.
\end{theorem}
\begin{proof} As in the proof of Theorem \ref{trm: Weakly_lipschitz_triang}, without loss of generality we can
assume that $A$\ is compact. Then we will prove the conclusion with $H$\ a Lipschitz mapping.
We proceed by induction on dimension of the set $A$. Let $\dim A = d$.

The case $d =0$\ is trivial. In case $d=1$\ to get the desired triangulation it suffices to
apply Theorem \ref{trm: Weakly_lipschitz_triang} (Corollary \ref{rem: lipschitz triangulation for compact sets}).
Let $d >1$\ and the theorem be true for the sets of dimension $\leq d-1$.
\vskip 2 mm

\textbf{Step 1.} By Theorem \ref{trm: Weakly_lipschitz_triang} (Corollary \ref{rem: lipschitz triangulation for compact sets}) we find
a definable $C^{q}$\ triangulation $(K_{1}, h_{1})$\ of the set $A$,
compatible with $A_{1},..., A_{r}$, Lipschitz and weakly bi-Lipschitz of class $C^{q}$\ on the simplicial stratification
$K_{1}$\ of the polyhedron $|K_{1}|$.
\vskip 2 mm

\textbf{Step 2.} By Theorem \ref{trm: Q invariance}, where we put $f=h_{1}, B=|K_{1}|, \sX_{B}=K_{1}$,
we find a definable $C^{q}$\ substratification $\sX'_{|K_{1}|}$\ of the polyhedron $|K_{1}|$\ that is compatible with  $K_{1}$\
 and such that the condition $\Q$\ and the pair $\left(h_{1},\sX'_{|K_{1}|}\right)$\ satisfy the conclusion of Theorem \ref{trm: Q invariance}.
In particular
$$\{\Gamma\in\sX'_{|K_{1}|}:\quad \dim\Gamma=d\}=\{\triangle\in K_{1}:\quad \dim\triangle=d\}.$$
Observe also, that the following family of definable $C^{q}$\ submanifolds
$$\sX'_{\left|K^{(d-1)}_{1}\right|}=\left\{\Gamma\in\sX'_{|K_{1}|}: \dim \Gamma\leq d-1\right\}$$
is a finite definable $C^{q}$\ stratification of the polyhedron $\left|K^{(d-1)}_{1}\right|$, compatible with
its natural simplicial stratification
$$K^{(d-1)}_{1}=\left\{\triangle\in K_{1}:\text{  }\dim \triangle\leq d-1 \right\}.$$
\vskip 2 mm

\textbf{Step 3.} Consider the polyhedron  $\left|K^{(d-1)}_{1}\right|$\ and its stratification $\sX'_{\left|K^{(d-1)}_{1}\right|}$.
As $\dim \left|K^{(d-1)}_{1}\right|<d$, then by the induction hypotesis there exists a definable $C^{q}$\ triangulation $(K_{2},h_{2})$\
of the polyhedron $\left|K^{(d-1)}_{1}\right|$,\ such that the family $\{h_{2}(\triangle):\text{   }\triangle\in K_{2}\}$\
forms a definable $C^{q}$\ stratification with the condition $\Q$\ of the set $\left|K^{(d-1)}_{1}\right|$,
compatible with the stratification $\sX'_{\left|K^{(d-1)}_{1}\right|}$\ and
$h_{2}: \left|K_{2}\right|\longrightarrow \left|K^{(d-1)}_{1}\right|$\ is a Lipschitz mapping.
\vskip 2 mm

\textbf{Step 4.} We define\footnote{Compare the construction in \cite{DG} Theorem 4.2, also \cite{ES} Example 9.5.3.} now a new simplicial complex $K_{3}$\ :
\vskip 2 mm
Let $\{a_{1},...,a_{\alpha}\}$\ be the set of all vertexes of $K_{2}$\ and let $\{\triangle_{\alpha+1},...,\triangle_{T}\}$\ be the set of all
$d$\ - dimensional simplexes of the complex $K_{1}$. Denote by $\{0_{\triangle_{\alpha+1}},...,0_{\triangle_{T}}\}$\ the set of
the barycentres of the simplexes $\triangle_{\alpha+1},...,\triangle_{T}$.
First we define the set of vertexes of $K_{3}$:
$$Vert(K_{3})=\{e_{1},...,e_{\alpha}, e_{\alpha+1},...,e_{T}\},$$
where $e_{j}$, $j=1,...,T$\ are the vertexes in $\R^{T}$\ of a standard $T-1$\ dimensional simplex
$$\triangle^{T-1}=[e_{1},...,e_{\alpha}, e_{\alpha+1},...,e_{T}].$$
Now we define the simpicial complex $K_{3}$\ as the following subcomplex of the natural simplicial complex
$K_{\triangle^{T-1}}=\left\{[e_{\delta_{1}},...,e_{\delta_{s}}]:\quad \delta_{1},...,\delta_{s}\in\{1,...,T\}\right\}$:
$$
\begin{tabular}{l c l}
$\triangle\in K_{3}$ & $\Longleftrightarrow$ &$ \triangle =[e_{\beta_{1}},...,e_{\beta_{s}}]\quad \text{for some}
\quad [a_{\beta_{1}},...,a_{\beta_{s}}]\in K_{2}$, \\
& or & $\triangle =[e_{j},e_{\beta_{1}},...,e_{\beta_{s}}]\quad \text{for some}\quad [a_{\beta_{1}},...,a_{\beta_{s}}]\in K_{2}\text{  and  }
\triangle_{j}, \text{ where }j\in\{\alpha+1,...,T\} \text{ is }$\\
&& \hskip 8,5 cm $\text{ such that }h_{2}([a_{\beta_{1}},...,a_{\beta_{s}}])\subset\overline{\triangle}_{j}$, \\
& or & $\triangle =[e_{j}]\quad\text{for some}\quad j=\alpha+1,...,T$.
\end{tabular}$$
Observe that $K_{3}$\ is a $d$\ - dimensional simplicial complex in $\R^{T}$\ and there exists a $d-1$\ dimensional subcomplex $L\subset K_{3}$:
$$L=\{[e_{\beta_{1}},..., e_{\beta_{s}}]:\quad [a_{\beta_{1}},...,a_{\beta_{s}}]\in K_{2}\},$$
simplicially isomorphic to $K_{2}$\ by the following semilinear homeomorphism $f: |L|\longrightarrow|K_{2}|$:
$$f\left(\underset{j=1}{\overset{s}{\sum}}\lambda_{j}\cdot e_{\beta_{j}}\right)=\underset{j=1}{\overset{s}{\sum}}\lambda_{j}\cdot a_{\beta_{j}},$$
for any $\underset{j=1}{\overset{s}{\sum}}\lambda_{j}\cdot e_{\beta_{j}}\in [e_{\beta_{1}},...,e_{\beta_{s}}]\in L$.
Then $\left(L, h_{2}\circ f\right)$\
is still a definable $C^{q}$\ triangulation of $\left|K_{1}^{(d-1)}\right|$\ with the condition $\Q$, compatible with 
$\sX'_{\left|K_{1}^{(d-1)}\right|}$\ and $h_{2}\circ f$\ is also a Lipschitz mapping.
\vskip 2 mm

\textbf{Step 5.} Now we extend the mapping $h_{2}\circ f$\ on the polyhedron $\left|K_{3}\right|$\ in a conical way:
$$h_{3}: \left|K_{3}\right|\longrightarrow \left|K_{1}\right|$$
$$h_{3}(z)=\begin{cases}h_{2}\circ f(z), & \text{ when }z\in\triangle',\text{   } \triangle'\in L,\text{   } \\
(1-t)\cdot 0_{\triangle_{j}}+t\cdot h_{2}\circ f(x), & \text{ when }z\in \triangle, \text{ }  \triangle\in K_{3}\setminus L,\text{  }\triangle=[e_{j},e_{\beta_{1}},...,e_{\beta_{s}}],
\\
 & z=(1-t)\cdot e_{j}+t\cdot x,\text{ } x\in [e_{\beta_{1}},...,e_{\beta_{s}}], \text{ } t\in(0,1),\\
0_{\triangle_{j}}, & \text{ when }z=e_{j},\text{ } j\in\{\alpha+1,...,T\} . \end{cases}$$
It follows from the construction that $h_{3}$\ has the properties:
\vskip 1 mm
\begin{tabular}{l}
$i)\quad h_{3}|_{\triangle}\text{  is a definable } C^{q} \text{  embedding for any } \triangle\in K_{3}$\ ; \\
$ii)\quad h_{3}\left(\left(e_{j},e_{\beta_{1}},...,e_{\beta_{s}}\right)\right)=
0_{\triangle_{j}}\ast h_{2}\circ f\left(\left(e_{\beta_{1}},...,e_{\beta_{s}}\right)\right),\quad
\text{for any }\triangle=[e_{j},e_{\beta_{1}},...,e_{\beta_{s}}]\in K_{3}\setminus L,$ \\
$iii)\quad \{h_{3}(\triangle):\text{  }\triangle\in L\}\text{ is a }C^{q} \text { stratification with the condition }\Q
 \text{ of
} \left|K^{(d-1)}_{1}\right|,$\ compatible with\\
\hskip 11,4 cm $ \text{ the stratification }\sX'_{\left|K^{(d-1)}_{1}\right|},$ \\
$iv)\quad h_{3}\ \text{ is a Lipschitz mapping}\footnotemark.$\\
\end{tabular}\\
\footnotetext{It suffices to observe that for any
 $\triangle\in K_{3}$\ the first order derivatives of $h_{3}|_{\triangle}$\ are bounded.}
By the conical property and Remark \ref{rem: generic for 0 dimensional} the family
$\{h_{3}(\triangle):\text{   }\triangle\in K_{3}\}$\ is a definable $C^{q}$\ stratification
with the\ condition $\Q$\ of the set $\left|K_{1}\right|$\ compatible with the stratification $\sX'_{|K_{1}|}$.
Now it is clear that
$$\left\{h_{1}\circ h_{3}(\triangle): \text{  }\triangle \in K_{3}\right\}$$
is a definable $C^{q}$\ stratification with the\ condition $\Q$\ of the set $A$, compatible with
$\left\{h_{1}(\triangle): \text{   } \triangle\in K_{1}\right\}$. It is also compatible with $A_{1},..., A_{r}$, because
the family $\left\{h_{1}(\triangle): \text{   } \triangle\in K_{1}\right\}$\ is compatible with $A_{1}$,...,$A_{r}$\ (see Step 1).
Hence, $\left(K_{3}, h_{1}\circ h_{3}\right)$\ is the desired definable $C^{q}$\ triangulation.

\end{proof}

\section{The Whitney (B) and the Verdier conditions as $\T$\ conditions}
\label{sec: Whitney Verdier are T}
In this section, using the results from \cite{Cz}, we will show that the Whitney (B) and the Verdier condition belong to the class of $\T$\ conditions.
First remind the basic definitions.

\begin{definition}\label{def: Whitney (B)} Let $N$, $M$\ be $C^{q}$\ submanifolds of $\R^{n}$\ ($q\geq 1$) such that
$N\subset \overline{M}\setminus M$\ and let $a\in N$.
We say that the pair ($M$, $N$)\ satisfies \textit{the Whitney (B)}\ \textit{condition at the point} $a$\ (notation: $\W^{B}(M,N,a)$)\ if
for any sequences $\{a_{\nu}\}_{\nu\in \N}\subset N$, $\{b_{\nu}\}_{\nu\in\N}\subset M$\
both converging to $a$\ and such that the sequence of the secant lines $\{\R(a_{\nu}-b_{\nu})\}_{\nu\in\N}$
\ converges to a line $L\subset \R^{n}$\ in $\mathbb{P}_{n-1}$\ and the sequence of the tangent spaces $\{T_{b_{\nu}}M\}_{\nu\in\N}$\
converges to a subspace $T\subset \R^{n}$\ in $\mathbb{G}_{\dim M,n}$, we have always $L\subset T$.
\vskip 1 mm
When a pair of $C^{q}$\ submanifolds $(M,N)$\ satisfies (respectively, does not satisfy) the Whitney (B) condition\ at a point $a\in N$, we write $\W^{B}(M,N,a)$\
(respectively $\sim\W^{B}(M,N,a)$). If for any point $a\in N$\ we have $\W^{B}(M,N,a)$, we write $\W^{B}(M,N)$.
\end{definition}

\begin{definition}
\label{def: d(v,W)}
Let $v\in S^{n-1}$\ and let $W$\ be a nonzero linear subspace of $\R^{n}$. We put
$$d(v, W)=\inf\{sin(v,w): w\in W\cap S^{n-1}\},$$
where $sin(v,w)$\ denotes the sine of the angle between the vectors $v$\ and $w$\ and $S^{n-1}=\{v\in\R^{n}: |v|=1\}$.
We also put $d(u,W)=1$, if $W=\{0\}$.
\end{definition}

\begin{definition}
\label{Funkcja d}
For any $P\in\mathbb{G}_{k,n}$\ and $Q\in\mathbb{G}_{l,n}$, we put
$$d(P,Q)=\sup\{d(\lambda; Q): \lambda \in P\cap S^{n-1}\}, $$
when $k>0$, and $d(P,Q)=0$, when $k=0$.
\end{definition}

Now we list some elementary properties of the function $d$, leaving the proof to the reader.

\begin{proposition}\label{prop: wlasnosci d}\ \\
a) Consider the following metric on $\mathbb{P}_{n-1}$:
\begin{center}
 $\widetilde{d}(\R v, \R w)=\min \{|u-w|, |u+w|\}$\hskip 3 mm for \hskip 2 mm $u,w\in S^{n-1}$.\\
 \end{center}
Then we have
$$\frac{1}{\sqrt{2}}\text{ }\widetilde{d}(\R v, \R w)\leq d(\R v, \R w)\leq \widetilde{d}(\R v, \R w).$$
b) If $V$, $W$\ are linear subspaces of $\R^{n}$, then $d(V\times\R, W\times\R)=d(V,W)$. \\
c) If $Q'\subset Q$, then $d(P,Q)\leq d(P,Q').$\\
d) For any $k\in\N$, $k\leq n$\ the function $d$\ is a metric on $\mathbb{G}_{k,n}$.
\end{proposition}

\begin{definition}\label{def: Verdier condition}
Let $\Lambda$, $\Gamma$\ be definable $C^{2}$\ submanifolds of $\R^{n}$, $\Gamma\subset\overline{\Lambda}\setminus\Lambda$. We say that
the pair $(\Lambda,\Gamma)$\ satisfies \textit{the Verdier condition at a point }$x_{0}\in\Gamma$\ (notation: $\W^{V}(\Lambda, \Gamma,x_{0})$), if
there exists an open  neighbourhood $U_{x_{0}}$\ of $x_{0}$\ in $\R^{n}$\ and $C_{x_{0}}>0$\ such that
 $$\forall x\in\Gamma\cap U_{x_{0}}\text{   }\forall y\in \Lambda\cap U_{x_{0}}\quad d\left(T_{x}\Gamma, T_{y}\Lambda\right)\leq C_{x_{0}}|x-y|.$$

In case the pair of submanifolds $(\Lambda,\Gamma)$\ satisfies the Verdier condition\
at each point $x_{0}\in \Gamma$, we write $\W^{V}(\Lambda,\Gamma)$.
\end{definition}

In the paper \cite{Cz}\ it was proved that for any o-minimal structure on real ordered field $\R$, admitting definable $C^{q}$\ Cell Decomposition\ with
$q\in\N\cup\{\infty,\omega\}$, $q\geq 1$, we have the following theorems.

\begin{theorem} The Whitney (B) condition is a $\WL$\ condition of class $C^{q}$, $q\geq 1$.
Also the Verdier condition is a $\WL$\ condition of class $C^{q}$, $q\geq 2$.
\end{theorem}

\begin{corollary}\label{col: Whitney Verdier are weakly lipschitz invariant}
Theorem \ref{trm: Q invariance} holds true for the Whitney (B) condition with $q\in\N\cup\{\infty,\omega\}$, $q\geq 1$\ and for the Verdier condition with $q\geq 2$.
\end{corollary}

\begin{theorem} The Whitney (B) condition has the conical property of class $C^{q}$, $q\geq 1$.
\end{theorem}

\begin{proof} As the Whitney (B) condition is $C^{q}$\ invariant\ with $q\geq 1$, we may use
 the equivalent definition from Remark \ref{rem: conical property 1}. Let $M$, $N$\ be definable $C^{q}$\ submanifolds of $\R^{n}$, $N\subset\overline{M}\setminus M$. Assume that $\W^{B}(M, N)$.
\vskip 1 mm
a) The conditions $\W^{B}\left(M\times (0;1),M\times\{1\}\right)$\ and $\W^{B}(N\times(0;1),N\times\{1\})$\ hold true, because
$M\times(0,1)\cup M\times\{1\}$\ and $N\times(0,1)\cup N\times\{1\}$\
are definable submanifolds with boundaries of class $C^{q}$, so they trivially satisfy the Whitney (B) condition.
\vskip 1 mm
b) Now we are to prove that $\W^{B}( M\times(0;1), N\times(0;1))$. Let $x\in N\times(0,1)$\ and consider two
 sequences $\{x_{\nu}\}_{\nu\in\N}\in N\times (0;1)$, $\{y_{\nu}\}_{\nu\in\N}\in M\times(0;1)$\ such that
$$x_{\nu}, y_{\nu}\longrightarrow x, \qquad \R(x_{\nu}-y_{\nu})\longrightarrow L,\qquad T_{y_{\nu}}(M\times (0;1))\longrightarrow T$$
for $\nu\longrightarrow +\infty$, with some $L\in\mathbb{P}_{n}$, $T\in\mathbb{G}_{\dim M+1,n+1}$. We have to show that $L\subset T$.
Denote the coordinates of the points $x$\ and $x_{\nu}, y_{\nu}$\ for $\nu\in\N$ :
$$x=(x',x_{n+1}), \qquad x_{\nu}=(x'_{\nu},x_{\nu n+1}),\qquad  y_{\nu}=(y'_{\nu},y_{\nu n+1})$$
with $x'$, $x'_{\nu}\in N$\ and $y'_{\nu}\in M$, $x_{n+1}$, $x_{\nu n+1}$, $y_{\nu n+1}\in (0,1)$. Then for any $\nu\in\N$, we
have the following inclusions
$$(\mathcal{S}1)\qquad \R (x_{\nu}-y_{\nu})=\R (x'_{\nu}-y'_{\nu}, x_{\nu n+1}-y_{\nu n+1})\text{\ \ }\subset \text{\ \ } \R (x'_{\nu}-y'_{\nu})\times
 \R(x_{\nu n+1}-y_{\nu n+1})\subset \R (x'_{\nu}-y'_{\nu})\times \R.$$
Moreover,
$$(\mathcal{S}2)\qquad \qquad \qquad \qquad \qquad \qquad T_{y_{\nu}}(M\times (0;1))=T_{y'_{\nu}}M\times \R. \qquad \qquad \qquad\qquad \qquad \qquad\qquad \qquad \qquad\qquad$$
As the sequences $\{x'_{\nu}\}_{\nu\in\N}$, $\{y'_{\nu}\}_{\nu\in\N}$\ tend to $x'$, so without loss of generality we may assume that
$$\R (x'_{\nu}-y'_{\nu})\longrightarrow L',\qquad T_{y'_{\nu}}M\longrightarrow T_{M}\qquad \text{for}\quad\nu\longrightarrow +\infty,$$
where $L'\in\mathbb{P}_{n-1}$, $T_{M}\in \mathbb{G}_{\dim M,n}$\ are some subspaces. From $(\mathcal{S}1)$\ and $(\mathcal{S}2)$\ we have
$$L\subset L'\times \R,\qquad T=T_{M}\times\R. $$
As $\W^{B}(M,N)$,\ then $L'\subset T_{M}$.
Finally,
$$L \quad \subset\quad L'\times \R\quad \subset\quad T_{M}\times \R = T.$$

c) The proof of $W^{B}( M\times(0;1), N\times \{1\})$\ is similar to the proof of Case b).

\end{proof}

\begin{corollary}
The Whitney (B) condition belongs to the class $\T$\ with $q\geq 1$.
\end{corollary}

In order to prove the conical property for the Verdier condition, we need one more lemma.

\begin{lemma}\label{lem: varphi is d Lipschitz } Let $k,n \in\N$, $k\leq n$\ and consider a linear subspace $E\subset\R^{n}$, $\dim E =k$. Let
$\mathcal{L}(E,E^{\perp})$\ be a vector space of linear mappings with the norm $||l||=\sup\{|f(v)|:\text{  }v\in E, |v|=1\}$. Consider the grassmanian
$\mathbb{G}_{k,n}$\ with the metric $d$\ and the mapping
$$\varphi: \mathcal{L}(E,E^{\perp})\ni l\longmapsto \widehat{l}\in\mathbb{G}_{k,n},\qquad \text{where}\qquad \widehat{l}=\{v+l(v):\quad v\in E\}.$$
Then $\varphi$\ is Lipschitz and
$$\forall \text{ } f,g\in\mathcal{L}(E,E^{\perp})\quad d\left(\widehat{f},\widehat{g}\right)\leq 2\cdot ||f-g||.$$
\end{lemma}
\begin{proof}
Let $v\in E$, $|v|=1$. Consider arbitrary functions $f, g\in \mathcal{L}(E,E^{\perp})$. Then
$$d\left(\R\frac{v+f(v)}{|v+f(v)|}, \widehat{g}\right)\overset{Prop. \ref{prop: wlasnosci d}c)}{\leq} d\left(\R \frac{v+f(v)}{|v+f(v)|}, \R \frac{v+g(v)}{|v+g(v)|}\right)
\overset{\text{Prop.}\ref{prop: wlasnosci d}a)}{\leq}\left|\frac{v+f(v)}{\sqrt{1+|f(v)|^{2}}}-\frac{v+g(v)}{\sqrt{1+|g(v)|^{2}}}\right|=$$
$$=\left|\frac{v+f(v)}{\sqrt{1+|f(v)|^{2}}}-\frac{v+g(v)}{\sqrt{1+|f(v)|^{2}}}+
[v+g(v)]\cdot \left[\frac{1}{\sqrt{1+|f(v)|^{2}}}-\frac{1}{\sqrt{1+|g(v)|^{2}}}\right]\right|=$$
$$= \left|\frac{f(v)-g(v)}{\sqrt{1+|f(v)|^{2}}}+[v+g(v)]\cdot \frac{\sqrt{1+|g(v)|^{2}}- \sqrt{1+|f(v)|^{2}}}
{\sqrt{1+|f(v)|^{2}}\cdot \sqrt{1+|g(v)|^{2}}}\cdot
\frac{\sqrt{1+|g(v)|^{2}}+ \sqrt{1+|f(v)|^{2}}}{\sqrt{1+|g(v)|^{2}}+ \sqrt{1+|f(v)|^{2}}} \right|\leq$$
$$\leq \left|\frac{f(v)-g(v)}{\sqrt{1+|f(v)|^{2}}}\right|+\sqrt{1+|g(v)|^{2}}\cdot \frac{||g(v)|^{2}- |f(v)|^{2}|}{\sqrt{1+|f(v)|^{2}}\cdot \sqrt{1+|g(v)|^{2}}}
\cdot \frac{1}{\sqrt{1+|f(v)|^{2}}+ \sqrt{1+|g(v)|^{2}}} \leq$$
$$\overset{\text{ triangle inequality}}{\leq} |f(v)-g(v)|\cdot\frac{1}{\sqrt{1+|f(v)|^{2}}}\cdot
\left(1+\frac{|f(v)|+|g(v)|}{\sqrt{1+|f(v)|^{2}}+ \sqrt{1+|g(v)|^{2}}}\right)\leq$$
$$\leq  2\cdot |f(v)-g(v)|\leq 2\cdot ||f-g||.$$
Therefore
$$d\left(\widehat{f},\widehat{g}\right)=\underset{v\in E\cap S^{n-1}}{\sup} d\left(\R\frac{v+f(v)}{|v+f(v)|}, \widehat{g}\right)\leq 2\cdot ||f-g||.$$
\end{proof}

\begin{corollary}\label{col: d(TxM, TyM)<C|x-y|} Let $\Lambda$\ be a definable $C^{q}$\ submanifold in $\R^{n}$, $q\in\N\cup\{\infty,\omega\}$, $q\geq 2$. Then for every point $x_{0}\in\Lambda$\ there exists
a constant $C_{x_{0}}>0$\ and an open neighbourhood $U_{x_{0}}$\ of the point $x_{0}$\ such that
$$\forall \text{  }x,y\in U_{x_{0}}\cap \Lambda\quad d\left(T_{x}\Lambda,T_{y}\Lambda\right)\leq C_{x_{0}}\cdot |x-y|.$$
\end{corollary}

Now we can prove the conical property for the Verdier condition.

\begin{theorem}\label{trm: Verdier is conical} The Verdier condition has the conical property of class $C^{q}$, $q\geq 2$.
\end{theorem}

\begin{proof} As the Verdier condition is $C^{2}$\ invariant, we may use the equivalent definition of the conical
property from Remark \ref{rem: conical property 1}. Consider two definable $C^{q}$\ ($q\geq 2$)\ submanifolds $M$, $N$\ of $\R^{n}$,
$N\subset\overline{M}\setminus M$. Assume that $\W^{V}(M,N)$.
\vskip 1 mm

a) We would like to prove that $\W^{V}\left(M\times (0;1),M\times\{1\}\right)$\ and $\W^{V}(N\times(0;1),N\times\{1\})$.
 Take a point $x_{0}\in M\times\{1\}$, $x_{0}=(x'_{0},1)$\ with some $x'_{0}\in M$.
 By Corollary \ref{col: d(TxM, TyM)<C|x-y|} we find an open neighbourhood $U_{x'_{0}}$\ of the point $x'_{0}$\ in $\R^{n}$\ and a constant
 $C_{x'_{0}}>0$\ such that
 $$\forall \text{  }x',y'\in U_{x'_{0}}\cap M\quad d\left(T_{x'}M,T_{y'}M\right)\leq C_{x'_{0}}\cdot |x'-y'|.$$
We claim that the following neighbourhood $U_{x_{0}}$\ and a constant $C_{x_{0}}$\  satisfy the thesis:
$$U_{x_{0}}=U_{x'_{0}}\times (0;1+\varepsilon)\qquad \text{and}\qquad C_{x_{0}}=C_{x'_{0}}, $$\
where $\varepsilon>0$. Let
 $x\in U_{x_{0}}\cap (M\times \{1\})$, $y\in U_{x_{0}}\cap (M\times (0;1))$. Then $x=(x',1)$, $y=(y',y_{n+1})$\ with some points
  $x'$, $y'\in U_{x'_{0}}\cap M$\ and $y_{n+1}\in (0;1)$.
 Moreover,
 $$T_{x}(M\times\{1\})=T_{x'}M\times \{0\}\qquad \text{ and }\qquad T_{y}(M\times (0;1))= T_{y'}M\times \R.$$
 Therefore
 $$d\left(T_{x'}M\times \{0\}, T_{y'}M\times \R\right)\overset{Prop. \ref{prop: wlasnosci d} c)}{\leq}
 d\left(T_{x'}M\times\{0\},T_{y'}M\times\{0\} \right)=$$
 $$\qquad \qquad \qquad \qquad \qquad \qquad  \qquad \qquad\qquad \qquad \quad \quad=d\left(T_{x'}M, T_{y'}M\right)\leq C_{x'_{0}}\cdot |x'-y'|\leq C_{x_{0}}\cdot |x-y|.$$
 The proof of $\W^{V}(N\times (0;1), N\times \{1\})$\ is exactly the same.
 \vskip 1 mm

b) Now we prove that $\W^{V}\left(M\times (0;1),N\times (0;1)\right)$.
 Take a point $x_{0}\in N\times(0;1)$, $x_{0}=(x'_{0},{x_{0}}_{n+1})$\ with some $x'_{0}\in N$, ${x_{0}}_{n+1}\in(0,1)$.
 By Corollary \ref{col: d(TxM, TyM)<C|x-y|} we find an open neighbourhood $U_{x'_{0}}$\ of the point $x'_{0}$\ in $\R^{n}$\ and a constant
 $C_{x'_{0}}>0$\ such that
 $$\forall \text{  }x'\in U_{x'_{0}}\cap N,y'\in U_{x'_{0}}\cap M\quad d\left(T_{x'}N,T_{y'}M\right)\leq C_{x'_{0}}\cdot |x'-y'|.$$
We claim that the following neighbourhood $U_{x_{0}}$\ and a constant $C_{x_{0}}$\  satisfy the thesis:
 $$U_{x_{0}}=U_{x'_{0}}\times (0;1)\qquad \text{and}\qquad C_{x_{0}}=C_{x'_{0}}.$$\
Let
 $x\in U_{x_{0}}\cap (N\times (0;1))$, $y\in U_{x_{0}}\cap (M\times (0;1))$.
Then
 $$d\left(T_{x}(N\times (0;1)), T_{y}(M\times (0;1))\right)=d\left(T_{x'}N\times \R, T_{y'}M\times \R\right)=$$
 $$\overset{\text{Prop.}\ref{prop: wlasnosci d}b)}{=} d(T_{x'}N,T_{y'}M)\leq C_{x'_{0}}\cdot |x'-y'|\leq C_{x'_{0}}\cdot |x-y|.$$
\vskip 1 mm

c) Take a point $x_{0}\in N\times\{1\}$, $x_{0}=(x'_{0},1)$\ with some $x'_{0}\in N$.
 By Corollary \ref{col: d(TxM, TyM)<C|x-y|} we find an open neighbourhood $U_{x'_{0}}$\ of the point $x'_{0}$\ in $\R^{n}$\ and a constant
 $C_{x'_{0}}>0$\ such that
 $$\forall \text{  }x'\in U_{x'_{0}}\cap N,y'\in U_{x'_{0}}\cap M\quad d\left(T_{x'}N,T_{y'}M\right)\leq C_{x'_{0}}\cdot |x'-y'|.$$
We claim that the following neighbourhood $U_{x_{0}}$\ and a constant $C_{x_{0}}$\  satisfy the thesis:
 $$U_{x_{0}}=U_{x'_{0}}\times (0;1+\varepsilon)\qquad \text{and}\qquad C_{x_{0}}=C_{x'_{0}},$$\
where $\varepsilon>0$. Let
 $x\in U_{x_{0}}\cap (N\times\{1\})$, $x=(x',1)$\ and $y\in U_{x_{0}}\cap (M\times (0;1))$, $y=(y',y_{n+1})$, where $y'\in M$, $y_{n+1}\in (0,1)$.
Then
 $$d\left(T_{x}(N\times\{1\}), T_{y}(M\times (0;1))\right)=d\left(T_{x'}N\times \{0\}, T_{y'}M\times \R\right)\leq $$
 $$\overset{\text{Prop.}\ref{prop: wlasnosci d}c)}{\leq} d(T_{x'}N\times\{0\},T_{y'}M\times\{0\})= d(T_{x'}N,T_{y'}M)\leq C_{x'_{0}}\cdot |x'-y'|\leq C_{x'_{0}}\cdot |x-y|.$$
\end{proof}

\begin{corollary}
The Verdier belongs to the class $\T$\ with $q\geq 2$.
\end{corollary}

\begin{remark}  The $(r)$\ condition of Kuo does not belong to the class $\T$, since it has not the conical property (see \cite{BT}).
\end{remark}
\vskip 5 mm

\textbf{Acknowledgements.} I am deeply grateful to Wies\l aw Paw\l ucki
for his friendly attention and many fascinating discussions. It is a great honour for me to have such a great teacher.

%%%%%%%%%%%%%%%%%%%%%%%%%%%%%%%%%%%%%%%%%%%%%%%%%%
%%%%%%%%%%%%%%%%%%%%%%%%%%%%%%%%%%%%%%%%%%%%%%%%%%
%References
%%%%%%%%%%%%%%%%%%%%%%%%%%%%%%%%%%%%%%%%%%%%%%%%%%
%\bibliographystyle{amsplain}

\end{document}